\documentclass[11pt,reqno]{amsart}
\usepackage[T1]{fontenc}
\usepackage{lmodern,microtype}
\usepackage{amsthm,amsmath,amsfonts,amssymb,mathtools,booktabs}
\usepackage[a4paper,margin=27mm]{geometry}
\usepackage[numbers,sort&compress]{natbib}
\usepackage{enumitem}
\setlist[enumerate]{label=\textnormal{(\roman*)},leftmargin=2em,itemsep=3pt}
\usepackage[hidelinks]{hyperref}
\hypersetup{pdftitle={Finite-state counterexamples to Ross's second queueing conjecture},pdfauthor={Yitzchak Shmalo}}
\numberwithin{equation}{section}
\theoremstyle{plain}
\newtheorem{theorem}{Theorem}[section]
\newtheorem{proposition}[theorem]{Proposition}
\newtheorem{lemma}[theorem]{Lemma}
\newtheorem{corollary}[theorem]{Corollary}
\theoremstyle{definition}
\newtheorem{remark}[theorem]{Remark}
\newcommand{\E}{\mathbb E}
\newcommand{\Pp}{\mathbb P}
\newcommand{\R}{\mathbb R}
\newcommand{\Z}{\mathbb Z}
\newcommand{\1}{\mathbf 1}
\newcommand{\ii}{\mathrm i}
\newcommand{\eps}{\varepsilon}
\DeclareMathOperator{\Var}{Var}
\DeclareMathOperator{\Cov}{Cov}
\DeclareMathOperator{\Unif}{Unif}
\DeclareMathOperator{\Rea}{Re}
\title[Counterexamples to Ross's second queueing conjecture]{Finite-state counterexamples to Ross's second queueing conjecture}
\author{Yitzchak Shmalo}
\address{Einstein Institute of Mathematics, The Hebrew University of Jerusalem, Jerusalem, Israel}
\email{Yitzchak.Shmalo@mail.huji.ac.il}
\date{19 September 2026}
\begin{document}
\begin{abstract}
We construct finite-state Markov arrival-rate processes for which faster modulation increases the mean stationary workload of a stable single-server queue, disproving the universal ordering in Leskel\"a's formulation of Ross's second conjecture. A 64-state example has positive transition rates between every pair of states, independent service times uniform on $[0.99,1.01]$, and an explicit positive workload gap. A 128-state example additionally has strictly positive arrival-rate autocovariance at every lag. More generally, finite-state examples can exhibit any prescribed finite number of separated workload rises and falls. The proofs use a second-order expansion with a nonnegative, speed-uniform remainder under only a finite second service moment. We characterize the service distributions for which the second-order workload coefficient is nonincreasing for every bounded stationary environment: the function $s\mapsto\E(1-\cos(sB))/s^2$ must be nonincreasing. Mixtures of exponential service distributions satisfy this criterion. An exact Palm identity transfers the workload reversals to customer waiting times.
\end{abstract}
\maketitle

\section{Introduction and main results}\label{sec:intro}
Consider a rate-one, work-conserving single-server queue. Conditional on a stationary nonnegative process $F=(F(t))_{t\in\R}$, its arrivals form a Poisson process with intensity
\begin{equation}\label{eq:arrival}
 \lambda_{\eps,c}(t)=\eps F(ct),\qquad \eps>0,\quad c>0.
\end{equation}
Service requirements are independent and identically distributed, independent of the environment and of the randomness generating arrivals. The parameter $\eps$ controls traffic, whereas $c$ changes only the time scale of the environment. Write $w_\eps(c)$ for the mean workload at an ordinary time in the stationary queue.

Ross \cite{Ross78} formulated two conjectures concerning queues with time-varying Poisson input. In the stationary Cox formulation, the first compares the queue with homogeneous Poisson input of the same mean rate; Rolski \cite{Rolski81} proved that comparison. The second concerns acceleration of the environment. In the journal version of his note, Leskel\"a \cite[Problem~1, p.~214]{Leskela22} asks whether
\begin{equation}\label{eq:ross}
 c_1\le c_2\quad\Longrightarrow\quad w_\eps(c_1)\ge w_\eps(c_2)
\end{equation}
holds for every stationary ergodic arrival-rate process, and, if not, which condition on that process is necessary and sufficient. We answer the universal assertion negatively. We also identify an exact criterion for its second-order light-traffic version, but not for the all-traffic ordering.

\subsection{An explicit counterexample}
Let $P_m$ be the clockwise cyclic shift on $\Z/m\Z$ and let $\Pi_m$ be averaging with respect to the uniform distribution. For $m\ge3$ and $\eta>0$, set
\begin{equation}\label{eq:generator}
 Q_{m,\eta}=\frac{P_m-I}{\sin(2\pi/m)}+\eta(\Pi_m-I),
 \qquad
 g_m(j)=2+\cos\left(\frac{2\pi j}{m}+\frac{\pi}{2m}\right).
\end{equation}
Here $P_mh(j)=h(j+1)$. If $X$ is the stationary chain with generator $Q_{m,\eta}$, then $F(t)=g_m(X_t)$ is itself a finite-state Markov chain: $g_m$ is injective. Its invariant distribution is uniform, $1<F<3$, and $\E F=2$. All off-diagonal transition rates are positive. The shift-plus-refresh form will make its covariance explicit.

\begin{theorem}[64 states and continuous service]\label{thm:64}
In \eqref{eq:generator}, take $m=64$ and $\eta=1/2048$. Use \eqref{eq:arrival} and independent service times
\begin{equation}\label{eq:uniformservice}
 B\sim\Unif[99/100,101/100].
\end{equation}
For every $0<\eps\le10^{-4}$, the stationary workload has a finite exponential moment and
\begin{equation}\label{eq:64main}
 w_\eps(3\pi)-w_\eps(2\pi)>\frac{\eps^2}{600}>0.
\end{equation}
The arrival-rate process is stationary, ergodic and Markov on 64 distinct positive values, and every off-diagonal transition rate is positive.
\end{theorem}

Thus the failure of \eqref{eq:ross} is not caused by nonstationary initialization, varying mean offered load, a degenerate service law, or forbidden transitions. The mean input work is $2\eps$ at both speeds. The lower bound in \eqref{eq:64main} is obtained by rational inequalities, not by estimating stationary queues through simulation.

The cyclic covariance oscillates in sign. The next result removes negative autocovariance as a necessary feature of the counterexample.
\begin{theorem}[Positive autocovariance at every lag]\label{thm:positive}
There is an explicit stationary irreducible 128-state Markov intensity $F_+$ with $1<F_+<5$, mean three and positive transition rates between every pair of states, such that
\begin{equation}\label{eq:poscovintro}
 \Cov(F_+(0),F_+(t))>0\qquad(t\in\R).
\end{equation}
With service law \eqref{eq:uniformservice}, its stationary workload mean satisfies
\begin{equation}\label{eq:positivegap}
 w^+_\eps(3\pi)-w^+_\eps(2\pi)>\frac{\eps^2}{500}>0,
 \qquad 0<\eps\le10^{-5}.
\end{equation}
The stationary workload has a finite exponential moment.
\end{theorem}
The complete generator and the covariance appear in Section~\ref{sec:positive}. Positive autocovariance is much weaker than an ordering of all finite-dimensional distributions; no association or supermodular ordering is asserted in Theorem~\ref{thm:positive}.

\subsection{Uniform expansion and structural consequences}
For a bounded stationary environment $0\le F\le M$ and a nonnegative service law with $b_1=\E B>0$ and $b_2=\E B^2<\infty$, define
\begin{equation}\label{eq:kernelintro}
 k_B(u)=\E(B-u)_+,
 \qquad
 A_B(c)=b_1\int_0^\infty k_B(u)\E[F(0)F(-cu)]\,du.
\end{equation}
Theorem~\ref{thm:expansion} proves the exact decomposition
\begin{equation}\label{eq:introexp}
 w_\eps(c)=\frac{\eps(\E F)b_2}{2}+\eps^2A_B(c)+R_{\eps,c},
 \qquad
 0\le R_{\eps,c}\le
 \frac{\eps^3M^3b_1^2b_2}{2(1-\eps Mb_1)},
\end{equation}
whenever $\eps Mb_1<1$. The error bound is uniform in $c$. Its proof compares the workload with the sum of the residual services that would remain under immediate parallel service. The difference of these two workloads is nonnegative and increases under insertion of a marked arrival.

The expansion also yields two structural conclusions. First, Theorem~\ref{thm:many} constructs, for any prescribed $N$, one finite-state environment and one continuous service law with $N$ distinct peaks in an explicit finite sequence of sampled workload means. Each sampled peak exceeds both its adjacent sampled values by a quantitative amount. Second, define
\begin{equation}\label{eq:Hintro}
 H_B(s)=\frac{\E(1-\cos(sB))}{s^2}\quad(s>0),
 \qquad H_B(0)=\frac{b_2}{2}.
\end{equation}
Theorem~\ref{thm:servicecriterion} proves that $A_B$ is nonincreasing for every bounded, mean-square-continuous stationary environment if and only if $H_B$ is nonincreasing. Testing only finite-state Markov intensities with all off-diagonal rates positive already gives the same equivalence. Any strict violation of this service criterion produces an actual workload reversal at sufficiently small positive traffic. This is a characterization of a uniform second-order comparison, not an all-traffic classification.

\subsection{Relation to the literature}
The exact formulation matters. Rolski's lower and upper comparisons \cite{Rolski81,Rolski86}, Heyman's discussion of Ross's original formulations \cite{Heyman82}, and the Ross-type comparisons of Miyoshi and Rolski \cite{MR04} belong to the history of this problem. The assertion addressed here is the ordering between two finite positive speeds in the stationary Cox formulation of Leskel\"a \cite{Leskela22}. We do not equate that assertion with comparison against homogeneous input, a frozen environment, or a transient average. Throughout the paper the service law is independent of the environment and the server works at a fixed rate.

Positive results under stronger dependence assumptions include the Markov-modulated comparisons of Chang, Chao and Pinedo \cite{CCP91} and B\"auerle and Rolski \cite{BR98}. The latter use doubly stochastically monotone environments. Dependence comparisons for stationary Markov processes and regularity based on directional convexity are developed in \cite{HuPan00,MS93}; related Markov orderings and stochastic relations appear in \cite{Massey87,Leskela10}. General treatments of these orders are \cite{MullerStoyan02,ShakedShanthikumar07}. Coupling and fixed-marginal comparison have classical foundations in \cite{Strassen65,Tchen80}. At the point-process level, comparison of counting processes and queues \cite{Whitt81} and directionally convex ordering of random measures and shot noise \cite{BY09} provide relevant frameworks. Our covariance calculation does not replace any of these stronger process-order hypotheses.

The use of light traffic is classical. Single-server approximations were developed by Daley and Rolski \cite{DR84,DR91,DR94}; general open-system expansions by Reiman and Simon \cite{RS89}; and factorial-moment and Markov-modulated expansions by B{\l}aszczyszyn and collaborators \cite{Blasz95,BFS95}. Last \cite{Last14} treats perturbations of Poisson processes in a general framework. We do not claim the principle of expanding in arrival rate as new. Here the useful feature is a nonnegative remainder with an elementary bound independent of modulation speed and valid under a finite second service moment. It turns a covariance computation into a strict finite-traffic inequality.

Cox input originates in \cite{Cox55}; point-process and random-measure foundations, including marking, thinning and Palm measures, are treated in \cite{DVJ03,DVJ08,Kallenberg17,LP18}. Markovian input models and their queueing applications include \cite{Neuts79,YN71,HL86}. Gupta, Harchol-Balter, Scheller-Wolf and Yechiali \cite{GHSY06} investigate fluctuating-load queues, including models with modulated service. Such models must be distinguished from the fixed-rate server and independent service requirements used here. Mixtures of exponentials, used below as a service class with a positive second-order comparison, also have a substantial approximation literature \cite{Whitt84}.

The stationary construction follows the reflected-input perspective of Lindley and Loynes \cite{Lindley52,Loynes62}; standard queueing background is \cite{Asmussen03}. We give the required moment and remainder arguments directly. The distinction between time and customer averages is classical \cite{Brumelle71,HS80}; it is especially important here because the Poisson-arrivals-see-time-averages principle \cite{Wolff82} cannot simply be applied to Cox input. Section~\ref{sec:palm} instead proves the required identity by marked compensation, consistent with Palm and martingale treatments \cite{Bremaud81,BB03} and rate-conservation methods \cite{Miyazawa94}.

The paper is organized as follows. Section~\ref{sec:expansion} proves the expansion and its finite-second-moment extension. Sections~\ref{sec:cycle} and \ref{sec:positive} prove the explicit counterexamples, and Section~\ref{sec:many} gives arbitrarily many sampled reversals. Section~\ref{sec:service} proves the service-law characterization; Section~\ref{sec:variance} discusses averaged-environment variance and reversible chains. Section~\ref{sec:palm} transfers the conclusions to customer waiting time.

\section{A uniform workload expansion}\label{sec:expansion}
We first work with an arbitrary jointly measurable stationary environment satisfying
\begin{equation}\label{eq:generalF}
 0\le F(t)\le M,\qquad a=\E F(0)>0.
\end{equation}
Let $B$ be a nonnegative service requirement with $\E B^2<\infty$, independent of all arrival and environmental randomness, and write
\begin{equation}\label{eq:moments}
 b_1=\E B>0,\qquad b_2=\E B^2,
 \qquad k_B(u)=\E(B-u)_+\quad(u\ge0).
\end{equation}
No Markov or mixing assumption is needed for the expansion. We assume the stronger, convenient stability condition
\begin{equation}\label{eq:domstability}
 \eps M b_1<1.
\end{equation}
All workloads in this section refer to the stationary version constructed from the infinite past.

\begin{theorem}[Uniform expansion under a finite second moment]\label{thm:expansion}
Under \eqref{eq:generalF}--\eqref{eq:domstability}, for every $c>0$,
\begin{equation}\label{eq:expansion}
 w_\eps(c)=\frac{\eps a b_2}{2}
 +\eps^2 b_1\int_0^\infty k_B(u)\E[F(0)F(-cu)]\,du+R_{\eps,c},
\end{equation}
where
\begin{equation}\label{eq:remainder}
 0\le R_{\eps,c}\le
 \frac{\eps^3 M^3 b_1^2 b_2}{2(1-\eps M b_1)}.
\end{equation}
The bound is uniform over all $c>0$ and all stationary environments satisfying \eqref{eq:generalF}. In particular, for $B\equiv1$,
\begin{equation}\label{eq:unitexpansion}
 w_\eps(c)=\frac{\eps a}{2}
 +\eps^2\int_0^1(1-u)\E[F(0)F(-cu)]\,du+R_{\eps,c},
 \quad 0\le R_{\eps,c}\le\frac{\eps^3 M^3}{2(1-\eps M)}.
\end{equation}
\end{theorem}
We first prove the theorem for bounded $B$, including all moment assertions needed for the balance identity. We then pass to general $B$ by monotone truncation. The stationary construction is the classical one associated with reflected input \cite{Loynes62,Asmussen03}.

\subsection{Stationary construction and moment bounds}
Let $\widehat N$ be a two-sided homogeneous Poisson process of rate $\rho=\eps M$, independent of the environment. Attach independent uniform marks on $[0,1]$ and independent service marks with law $B$ to its points. Keep a point at time $t$ when its uniform mark is at most $F(ct)/M$. Conditional on the whole environment path, the retained points form a Poisson process with intensity \eqref{eq:arrival}. Its service marks remain independent with law $B$.

For a locally finite marked arrival configuration $A=\{(T_i,B_i):i\in\mathcal I\}$, define the workload at time zero by
\begin{equation}\label{eq:reich}
 V(A)=\sup_{u\ge0}\left\{\sum_{-u<T_i\le0}B_i-u\right\}.
\end{equation}
The term $u=0$ is zero when there is no arrival at zero; deterministic times are almost surely not arrival times. Values at an arrival of age $v$ are obtained by taking $u\downarrow v$ from above. Formula \eqref{eq:reich} is also the increasing limit of the workloads obtained by starting empty at times tending to $-\infty$.

\begin{lemma}[Poisson domination for bounded service]\label{lem:domination}
Assume for this lemma that $B$ is bounded. Let $\widehat A$ be the full marked Poisson configuration before thinning, and put $\widehat V=V(\widehat A)$. Under \eqref{eq:domstability}, $\widehat V$ has a finite exponential moment. The retained configuration satisfies
\[
 0\le V(A)\le\widehat V\qquad\text{almost surely}.
\]
Consequently the construction gives a finite stationary workload with moments of all orders, uniformly in $c$.
\end{lemma}
\begin{proof}
On the reversed time axis, the total incoming work of $\widehat A$ during an interval of length $u$ is a compound Poisson process $J(u)$ with rate $\rho$ and jump law $B$. Since $B$ is bounded,
\[
 \psi(s)=\rho(\E e^{sB}-1)-s
\]
is finite for every $s\ge0$. It satisfies $\psi(0)=0$ and $\psi'(0)=\rho b_1-1<0$. Choose $s>0$ with $\psi(s)<0$. Independent increments show that $\exp\{s(J(u)-u)\}$ is a nonnegative supermartingale. Its maximal inequality, first on a finite interval and then by monotone convergence, yields
\[
 \Pp(\widehat V>x)\le e^{-sx},\qquad x>0.
\]
Hence $\E e^{s'\widehat V}<\infty$ for $0<s'<s$. The pathwise domination follows from \eqref{eq:reich}, because thinning removes nonnegative service requirements. The marked construction is jointly stationary with the environment, and \eqref{eq:reich} is translation-covariant. It therefore gives a stationary workload process. The same Poisson upper bound applies for every $c$.
\end{proof}

\subsection{The excess over parallel service}
Define the workload that would remain if each customer were served immediately by its own rate-one server:
\begin{equation}\label{eq:shot}
 S(A)=\sum_{T_i\le0}(B_i+T_i)_+.
\end{equation}
At this stage $B$ is bounded, so the sum is finite. Set
\begin{equation}\label{eq:D}
 D(A)=V(A)-S(A).
\end{equation}
We compare marked configurations by inclusion, retaining the service mark of every common point.

\begin{lemma}[Monotonicity of the excess]\label{lem:excess}
For finite marked configurations, $D(A)\ge0$, and adding an arrival never decreases $D$. More precisely, if an arrival with service requirement $b\ge0$ and age $v\ge0$ is added, then
\begin{equation}\label{eq:insertion}
 V\bigl(A\cup\{(-v,b)\}\bigr)-V(A)\ge(b-v)_+.
\end{equation}
The same conclusions hold for infinite configurations with bounded service marks and finite workload. In particular, under the thinning construction,
\begin{equation}\label{eq:Ddom}
 0\le D(A)\le D(\widehat A).
\end{equation}
\end{lemma}
\begin{proof}
It suffices first to prove \eqref{eq:insertion} for a finite configuration. If $b\le v$, the right side is zero and the result follows from monotonicity of \eqref{eq:reich}. Suppose $b>v$. Let
\[
 H_A(u)=\sum_{-u<T_i\le0}B_i-u.
\]
For $u>v$, the term $H_A(u)$ increases by $b$, which is at least $b-v$. For $0\le u\le v$ and any $\gamma>0$, all old arrivals counted by $H_A(u)$ are also counted in the interval of length $v+\gamma$. The latter interval contains the added arrival, so the new supremum is at least
\[
 H_A(u)+b-v+u-\gamma\ge H_A(u)+b-v-\gamma.
\]
Letting $\gamma\downarrow0$ and then taking the supremum over $u$ proves \eqref{eq:insertion}. The added point contributes exactly $(b-v)_+$ to \eqref{eq:shot}. Thus $D$ is nondecreasing under insertion. Since $D(\varnothing)=0$, it is nonnegative on every finite configuration.

For an infinite configuration, truncate the arrivals at time $-L$. The workloads increase to \eqref{eq:reich}, and the parallel-service sum is unchanged once $L$ exceeds a common bound on the service marks. Passing to the limit gives the assertion. Applying the finite comparison to the paired truncations of $A\subseteq\widehat A$ proves \eqref{eq:Ddom}.
\end{proof}

\subsection{Workload balance}
Let $V(t)$ denote the stationary workload process, with its right-continuous convention at arrivals. Between arrivals, $V^2$ has derivative $-2V$, including when $V=0$. At an arrival with mark $b$, its jump is $2V(t-)b+b^2$. Consequently, for $T>0$,
\begin{equation}\label{eq:pathbalance}
 V(T)^2-V(0)^2=-2\int_0^T V(t)\,dt
 +\sum_{0<T_i\le T}\bigl(2V(T_i-)B_i+B_i^2\bigr).
\end{equation}
Lemma~\ref{lem:domination} justifies expectations of all terms. Conditional on the whole environment, the arrival process is Poisson, and its service marks are independent of the prearrival workload. Compensation in \eqref{eq:pathbalance} and stationarity give
\begin{equation}\label{eq:balance}
 w_\eps(c)=\frac{\eps a b_2}{2}+b_1\E[\lambda_{\eps,c}(0)V(0)].
\end{equation}
Here $V(0-)=V(0)$ almost surely. One may use the filtration containing the whole environment path from time zero and the Poisson randomness revealed up to the present; the compensator remains $\lambda_{\eps,c}(t)\,dt$.

For the dominating homogeneous queue, the same identity has constant intensity $\rho$. It gives
\begin{equation}\label{eq:pk}
 \E\widehat V=\frac{\rho b_2}{2(1-\rho b_1)}.
\end{equation}
Taking expectations in \eqref{eq:shot} for that queue gives
\begin{equation}\label{eq:shotmean}
 \E\widehat S=\rho\int_0^\infty k_B(u)\,du=\frac{\rho b_2}{2}.
\end{equation}
The last equality follows by integrating $(B-u)_+$ and then taking its expectation. Subtracting \eqref{eq:shotmean} from \eqref{eq:pk},
\begin{equation}\label{eq:meanD}
 \E D(\widehat A)=\frac{\rho^2 b_1b_2}{2(1-\rho b_1)}.
\end{equation}
\begin{proof}[Proof of Theorem~\ref{thm:expansion} for bounded service]
Insert $V=S+D$ into \eqref{eq:balance}. Given the environment, the expected parallel-service workload is
\[
 \E[S\mid F]=\eps\int_0^\infty k_B(u)F(-cu)\,du.
\]
Multiplying by $\lambda_{\eps,c}(0)=\eps F(0)$ and taking expectations gives the second-order term of \eqref{eq:expansion}. Its remaining term is exactly
\begin{equation}\label{eq:Rexact}
 R_{\eps,c}=b_1\E[\lambda_{\eps,c}(0)D(A)].
\end{equation}
It is nonnegative by Lemma~\ref{lem:excess}. Since $\lambda_{\eps,c}(0)\le\eps M=\rho$, equations \eqref{eq:Ddom} and \eqref{eq:meanD} imply
\[
 R_{\eps,c}\le b_1\rho\,\E D(\widehat A)
 =\frac{\rho^3 b_1^2b_2}{2(1-\rho b_1)}.
\]
This proves \eqref{eq:remainder}. Setting $B\equiv1$ gives \eqref{eq:unitexpansion}.
\end{proof}

\subsection{Removing boundedness of the service law}
\begin{proof}[Completion of the proof of Theorem~\ref{thm:expansion}]
Suppose only that $b_2<\infty$. In the common marked Poisson construction, replace each mark $B_i$ by $B_i^{(n)}=B_i\wedge n$. Write $V_n$, $w_n$, $b_{1,n}$, $b_{2,n}$ and $k_n$ for the resulting quantities. For every fixed interval, its incoming work increases to that of the untruncated configuration. Interchanging two suprema gives
\[
 V_n\uparrow V:=\sup_{u\ge0}\left\{\sum_{-u<T_i\le0}B_i-u\right\}.
\]
For the dominating homogeneous queue, the bounded-service calculation and monotone convergence give
\[
 \E\widehat V
 =\lim_n\frac{\rho b_{2,n}}{2(1-\rho b_{1,n})}
 =\frac{\rho b_2}{2(1-\rho b_1)}<\infty.
\]
Consequently $V\le\widehat V<\infty$ almost surely and $w_n\uparrow w_\eps(c)<\infty$. The stationary construction is unchanged. In the coefficient formula, $b_{1,n}\uparrow b_1$, $k_n\uparrow k_B$ and the correlation integrand is nonnegative; hence $A_{B\wedge n}(c)\uparrow A_B(c)$. This limit is finite because $\int k_B=b_2/2$ and $F\le M$.

The bounded-service identity now implies convergence of
\[
 R_{n,c}=w_n-\frac{\eps a b_{2,n}}2-\eps^2A_{B\wedge n}(c)
\]
to the remainder in \eqref{eq:expansion}. Passing to the limit in its nonnegative upper and lower bounds proves \eqref{eq:remainder}. No second moment of the untruncated workload is assumed. A finite second service moment gives a finite workload mean; the exponential workload moment in the explicit examples follows separately from their bounded service laws.
\end{proof}
This argument uses monotone convergence and an integrable Poisson dominating workload, rather than an unsupported exchange of stationary limits. Uniform-integrability criteria in stochastic-order language are discussed more generally in \cite{LV13}.

\begin{corollary}[A quantitative sign test]\label{cor:signtest}
For a fixed service law, define
\begin{equation}\label{eq:coefficient}
 A_B(c)=b_1\int_0^\infty k_B(u)\E[F(0)F(-cu)]\,du.
\end{equation}
If $0<c_1<c_2$ and $\Delta=A_B(c_2)-A_B(c_1)>0$, then
\begin{equation}\label{eq:signlower}
 w_\eps(c_2)-w_\eps(c_1)
 \ge\eps^2\left(\Delta-
 \frac{\eps M^3b_1^2b_2}{2(1-\eps M b_1)}\right).
\end{equation}
In particular, the difference is positive whenever
\begin{equation}\label{eq:epsthreshold}
 0<\eps<\frac{2\Delta}{M^3b_1^2b_2+2\Delta M b_1}.
\end{equation}
\end{corollary}
\begin{proof}
The first-order term in \eqref{eq:expansion} is independent of $c$. Use $R_{\eps,c_2}\ge0$ and the upper bound on $R_{\eps,c_1}$. Solving the resulting strict inequality yields \eqref{eq:epsthreshold}, whose upper endpoint is smaller than $1/(M b_1)$.
\end{proof}

\section{The cyclic environment and the 64-state example}\label{sec:cycle}
For the family \eqref{eq:generator}, put
\begin{equation}\label{eq:delta}
 \theta=\frac{2\pi}{m},\qquad d=\tan\frac\pi m+\eta.
\end{equation}
\begin{lemma}\label{lem:cycle}
The stationary intensity $F(t)=g_m(X_t)$ has mean two, takes $m$ distinct values in $(1,3)$, and satisfies
\begin{equation}\label{eq:covariance}
 \E[F(0)F(t)]=4+\frac12e^{-d|t|}\cos t.
\end{equation}
It is an irreducible, mixing continuous-time Markov chain. If $\eta>0$, all its off-diagonal transition rates are positive.
\end{lemma}
\begin{proof}
Both $P_m$ and $\Pi_m$ preserve the uniform law. Clockwise transitions make the chain irreducible; a finite irreducible continuous-time chain is mixing. These finite-state facts may also be read from its transition semigroup; see \cite{Norris97} for background. Equality of two values of $g_m$ requires either $j-k\in m\Z$ or $j+k+1/2\in m\Z$. The latter is impossible, so $g_m$ is injective. The cosine never equals $1$ or $-1$, which also proves the strict intensity bounds.

Let $h(j)=\exp(\ii\theta j+\ii\theta/4)$. Then $P_mh=e^{\ii\theta}h$, $\Pi_mh=0$, and
\[
 Q_{m,\eta}h=\left(\frac{e^{\ii\theta}-1}{\sin\theta}-\eta\right)h=(-d+\ii)h.
\]
The uniform averages of $h$ and $h^2$ vanish, whereas $|h|=1$. Thus, for $t\ge0$, the covariance of the real parts is $\tfrac12e^{-dt}\cos t$. A real scalar stationary covariance is even, proving \eqref{eq:covariance} for all $t$. Finally, the refresh term gives $\eta/m>0$ to every off-diagonal entry.
\end{proof}
For $d\ge0$ and $c>0$, define
\begin{equation}\label{eq:Id}
 I_d(c)=\int_0^1(1-u)e^{-dcu}\cos(cu)\,du.
\end{equation}
For unit service in the cyclic environment,
\begin{equation}\label{eq:Aunit}
 A_1(c)=2+\frac12I_d(c).
\end{equation}
We need a lower bound on the difference at two particular speeds, not a numerical approximation to either queue.

\begin{lemma}[Endpoint comparison]\label{lem:endpointgap}
For $G(d)=\tfrac12(I_d(3\pi)-I_d(2\pi))$,
\begin{equation}\label{eq:Gexact}
 G(d)=-\frac{d}{12\pi(1+d^2)}
 +\frac{1-d^2}{72\pi^2(1+d^2)^2}
 \left(-5+4e^{-3\pi d}+9e^{-2\pi d}\right).
\end{equation}
If $0\le d\le1/20$, then
\begin{equation}\label{eq:Gbound}
 G(d)>\frac{363599}{115776720}>\frac1{320}.
\end{equation}
\end{lemma}
\begin{proof}
Direct integration gives, for $k=2,3$,
\[
 I_d(k\pi)=\frac{d}{k\pi(1+d^2)}
 +\frac{(1-d^2)(1-(-1)^ke^{-k\pi d})}{k^2\pi^2(1+d^2)^2}.
\]
Subtraction proves \eqref{eq:Gexact}. For $0\le d\le1/20$, use $e^{-x}\ge1-x$ and $\pi<22/7$ to obtain
\[
 -5+4e^{-3\pi d}+9e^{-2\pi d}\ge8-30\pi d>\frac{23}{7}.
\]
This bracket is positive. The bounds $\pi>3$, $\pi^2<10$ and $d^2\le1/400$ yield
\[
 G(d)>-\frac1{720}
 +\frac{399\cdot400}{720\cdot401^2}\frac{23}{7}
 =\frac{363599}{115776720}.
\]
The difference between this fraction and $1/320$ is $7187/463106880>0$.
\end{proof}

\begin{lemma}[Perturbing unit service]\label{lem:service}
Suppose $\E B=1$ and $\E B^2<\infty$. Then
\begin{equation}\label{eq:kernelmass}
 h_B(u):=k_B(u)-(1-u)_+\ge0,
 \qquad \int_0^\infty h_B(u)\,du=\frac{\Var(B)}2.
\end{equation}
If $C(t)=\E[F(0)F(t)]$ takes its values in an interval of length $L_C$, then for any $c_1,c_2>0$,
\begin{equation}\label{eq:servicegap}
 \left|A_B(c_2)-A_B(c_1)-A_1(c_2)+A_1(c_1)\right|
 \le L_C\frac{\Var(B)}2.
\end{equation}
For \eqref{eq:covariance}, one can take $L_C=1$.
\end{lemma}
\begin{proof}
Convexity of $b\mapsto(b-u)_+$ proves the first inequality. Integration gives $(\E B^2-1)/2=\Var(B)/2$. The expression in \eqref{eq:servicegap} equals
\[
 \int_0^\infty h_B(u)\bigl(C(c_2u)-C(c_1u)\bigr)\,du,
\]
whose absolute value is bounded as claimed. Formula \eqref{eq:covariance} lies in $[7/2,9/2]$.
\end{proof}

\begin{proof}[Proof of Theorem~\ref{thm:64}]
Write $d_* = \tan(\pi/64)+1/2048$. For $x_0=11/224$, the elementary bounds $\pi<22/7$, $\sin x\le x$ and $\cos x\ge1-x^2/2$ give
\begin{equation}\label{eq:64damping}
 d_*<\frac{x_0}{1-x_0^2/2}+\frac1{2048}
 =\frac{10192775}{205273088}<\frac1{20}.
\end{equation}
The last gap is $354397/1026365440>0$. The map $x/(1-x^2/2)$ is increasing on $[0,\sqrt2)$, which justifies substituting $x_0$.

The unit-service coefficient gap is therefore greater than $1/320$. For \eqref{eq:uniformservice},
\begin{equation}\label{eq:uniformmoments}
 b_1=1,\qquad \Var(B)=\frac1{30000},\qquad b_2=\frac{30001}{30000}.
\end{equation}
Lemma~\ref{lem:service} loses at most $1/60000$. Since $M=3$ and $0<\eps\le10^{-4}$,
\[
 \frac{27\eps}{2(1-3\eps)}\le\frac{27}{19994}<\frac1{700}.
\]
Corollary~\ref{cor:signtest} now gives
\begin{align}\label{eq:64fullmargin}
 \frac{w_\eps(3\pi)-w_\eps(2\pi)}{\eps^2}
 &>\frac1{320}-\frac1{60000}-\frac{30001}{21000000}\notag\\
 &=\frac{5879}{3500000}>\frac1{600}.
\end{align}
The final gap is $137/10500000>0$. Lemma~\ref{lem:cycle} verifies the environmental assumptions. The off-diagonal generator entries are
\[
 (Q_{64,1/2048})_{jk}
 =\frac{\1_{\{k=j+1\pmod{64}\}}}{\sin(\pi/32)}+\frac1{131072}>0.
\]
Bounded service and $3\eps<1$ give the exponential workload moment by Lemma~\ref{lem:domination}.
\end{proof}
\begin{remark}
No minimality of 64 states is asserted. Larger cycles also give explicit examples; the endpoint identity makes a separate large-state construction unnecessary.
\end{remark}

\section{A counterexample with positive autocovariance}\label{sec:positive}
Let $X$ be the 64-state chain above. Independently, let $Z$ be the stationary chain on $\{-1,1\}$ that flips sign at rate $\alpha/2$, where
\begin{equation}\label{eq:positiverates}
 \alpha=\frac1{10000},\qquad \zeta=\frac1{100000}.
\end{equation}
On the 128-state product space, add independent uniform refreshes at rate $\zeta$. More precisely, let $Y$ have generator
\begin{equation}\label{eq:positivegenerator}
 \mathcal Q=Q_{64,1/2048}\otimes I+I\otimes Q_Z
                  +\zeta(\Pi_{128}-I),
\end{equation}
and define
\begin{equation}\label{eq:positiveF}
 F_+(t)=3+z+\cos\left(\frac{2\pi j}{64}+\frac\pi{128}\right)
 \quad\text{when }Y_t=(j,z).
\end{equation}
The uniform law is invariant and every off-diagonal transition rate is at least $\zeta/128$. The values with $z=-1$ lie strictly between one and three; those with $z=1$ lie strictly between three and five. Within each block the cosine values are distinct by Lemma~\ref{lem:cycle}. Hence $F_+$ is itself a 128-state Markov chain.

\begin{lemma}\label{lem:positivecov}
For $d_*$ in \eqref{eq:64damping}, the mean of $F_+$ is three and
\begin{equation}\label{eq:positivecov}
 K_+(t):=\Cov(F_+(0),F_+(t))
 =e^{-(\alpha+\zeta)|t|}
       +\frac12e^{-(d_*+\zeta)|t|}\cos t.
\end{equation}
In particular,
\begin{equation}\label{eq:positivecovbound}
 \frac12e^{-(\alpha+\zeta)|t|}\le K_+(t)\le\frac32
 \qquad(t\in\R).
\end{equation}
\end{lemma}
\begin{proof}
The centered coordinate $z$ is an eigenfunction of the product generator with eigenvalue $-\alpha$, and the cyclic character used in Lemma~\ref{lem:cycle} has eigenvalue $-d_*+\ii$. Both have uniform mean zero, so the added refresh subtracts $\zeta$ from both eigenvalues. Their cross inner products vanish under the uniform law. The variance contributions are respectively one and one half, proving \eqref{eq:positivecov}. Since $d_*>3/64>\alpha$, the oscillatory term has absolute value at most $\tfrac12e^{-(\alpha+\zeta)|t|}$. This gives the lower bound; the upper bound follows by discarding both exponential decays.
\end{proof}

\begin{proof}[Proof of Theorem~\ref{thm:positive}]
The estimate \eqref{eq:64damping} remains below $1/20$ after adding $\zeta$, since
\[
 \frac1{20}-\frac{10192775}{205273088}-\frac1{100000}
 =\frac{215083341}{641478400000}>0.
\]
Thus the cyclic part of the unit-service coefficient gap still exceeds $1/320$. For the nonoscillating covariance term set $a_+=\alpha+\zeta=11/100000$ and
\[
 J(c)=\int_0^1(1-u)e^{-a_+cu}\,du.
\]
The inequality $1-e^{-x}\le x$ gives
\[
 J(3\pi)-J(2\pi)\ge-a_+\pi\int_0^1u(1-u)\,du
 >-\frac{121}{2100000}.
\]
By Lemma~\ref{lem:positivecov}, the uncentered correlation lies in an interval of length $3/2$. Lemma~\ref{lem:service} therefore loses at most $1/40000$ when service is changed to \eqref{eq:uniformservice}. Finally, $M=5$ and $\eps\le10^{-5}$ give
\[
 \frac{125\eps b_2}{2(1-5\eps)}
 \le\frac{30001}{47997600}<\frac1{1500}.
\]
Combining these estimates with Corollary~\ref{cor:signtest},
\begin{align}\label{eq:positivemargin}
 \frac{w^+_\eps(3\pi)-w^+_\eps(2\pi)}{\eps^2}
 &>\frac1{320}-\frac{121}{2100000}-\frac1{40000}-\frac1{1500}\notag\\
 &=\frac{1663}{700000}>\frac1{500}.
\end{align}
Stationarity, ergodicity, injectivity and positive rates were checked above. Bounded service and $5\eps<1$ give the exponential workload moment. Strict positivity of every finite-lag covariance follows from \eqref{eq:positivecovbound}.
\end{proof}

\section{Arbitrarily many sampled workload reversals}\label{sec:many}
The same mechanism is not limited to a single pair of speeds.
\begin{theorem}\label{thm:many}
For each integer $N\ge1$, put $L=N+1$ and choose
\begin{equation}\label{eq:manyparameters}
 m=1024L^3,\quad \eta=1/m,\quad
 B\sim\Unif[1-h,1+h],\quad h=\frac1{100L}.
\end{equation}
Use the stationary intensity \eqref{eq:generator}. For every $0<\eps\le1/(10000L^2)$ and $j=1,\ldots,N$,
\begin{align}\label{eq:manygaps}
 w_\eps((2j+1)\pi)-w_\eps(2j\pi)&>\frac{\eps^2}{100L^2},\notag\\
 w_\eps((2j+1)\pi)-w_\eps((2j+2)\pi)&>\frac{\eps^2}{100L^2}.
\end{align}
All off-diagonal environmental rates are positive, and the stationary workload has a finite exponential moment.
\end{theorem}
\begin{proof}
Let $d=\tan(\pi/m)+1/m$. The bounds $\tan x<2x$ for $0<x<\pi/4$ and $\pi<22/7$ imply $d<8/m$. Also
\[
 I_0(c)=\frac{1-\cos c}{c^2},\qquad
 |I_d(c)-I_0(c)|\le\frac{dc}{6},
\]
the second inequality following by integrating $(1-u)dcu$ over $[0,1]$.

Fix $j\le N$. At the odd speed the undamped coefficient exceeds that at either adjacent even speed by
\[
 \frac{1}{\pi^2(2j+1)^2}>\frac1{40L^2}.
\]
The total loss from damping, including the factor one half in \eqref{eq:Aunit}, is at most
\[
 \frac{d\pi(4j+3)}{12}<\frac{32L}{3m}=\frac1{96L^2}.
\]
The same upper bound covers the preceding even speed. By Lemma~\ref{lem:service}, changing the service law loses at most $\Var(B)/2=1/(60000L^2)$.

Since $b_2=1+1/(30000L^2)\le30001/30000$ and $\eps\le1/(10000L^2)$,
\[
 \frac{27\eps b_2}{2(1-3\eps)}<\frac{30001}{21000000L^2}.
\]
The sign test can compare any ordered pair of coefficient values; it does not require that the larger coefficient occur at the larger speed. For both differences in \eqref{eq:manygaps}, its normalized lower bound is greater than
\[
 \frac1{L^2}\left(\frac1{40}-\frac1{96}-\frac1{60000}
                         -\frac{30001}{21000000}\right)
 =\frac{275899}{21000000L^2}>\frac1{100L^2}.
\]
The remaining assertions follow from Lemmas~\ref{lem:domination} and \ref{lem:cycle}.
\end{proof}
\begin{remark}
The theorem gives $N$ strict peaks in a specified finite sequence of sampled means. It does not assert a derivative sign at every intervening speed, or infinitely many oscillations for a single fixed environment. The parameters depend on $N$.
\end{remark}

\section{A service-law characterization of second-order ordering}\label{sec:service}
We now separate the environmental spectrum from the service law. Let $F$ be bounded, stationary and mean-square continuous, let $a=\E F(0)$, and let $K(t)=\Cov(F(0),F(t))$. The spectral representation of a continuous stationary covariance \cite{Yaglom87} provides a finite symmetric positive measure $\mu$ such that
\begin{equation}\label{eq:spectral}
 K(t)=\int_{\R}\cos(\omega t)\,\mu(d\omega).
\end{equation}
No reversibility is required.

\begin{lemma}[The service spectral multiplier]\label{lem:servicefilter}
For $B\ge0$ with $0<b_1=\E B$ and $b_2=\E B^2<\infty$, define $H_B$ by \eqref{eq:Hintro} and extend it evenly to $\R$. Then $H_B$ is continuous, $0\le H_B\le b_2/2$, and
\begin{align}\label{eq:servicefilter}
 H_B(s)&=\int_0^\infty k_B(u)\cos(su)\,du,\notag\\
 A_B(c)&=\frac{a^2b_1b_2}{2}
                   +b_1\int_{\R}H_B(c\omega)\,\mu(d\omega).
\end{align}
\end{lemma}
\begin{proof}
For fixed $b\ge0$ and $s\ne0$, integration by parts gives
\[
 \int_0^b(b-u)\cos(su)\,du=\frac{1-\cos(sb)}{s^2}.
\]
The absolute integral is at most $b^2/2$, so Fubini applies after expectation. The inequality $1-\cos x\le x^2/2$ gives the bound on $H_B$, and dominated convergence gives continuity, including at zero. Insert \eqref{eq:spectral} into \eqref{eq:kernelintro}; the finite mass of $\mu$ and $\int k_B=b_2/2$ justify the second use of Fubini.
\end{proof}

\begin{theorem}[Universal second-order service criterion]\label{thm:servicecriterion}
Fix a nonnegative service distribution with $0<\E B$ and $\E B^2<\infty$. The following are equivalent:
\begin{enumerate}
\item $H_B(s)$ is nonincreasing on $[0,\infty)$.
\item For every bounded nonnegative mean-square-continuous stationary environment of positive mean, $A_B(c)$ is nonincreasing on $(0,\infty)$.
\item The preceding assertion holds for every strictly positive stationary finite-state Markov intensity with positive transition rates between every pair of distinct states.
\end{enumerate}
If these conditions fail, a finite-state environment of the class in (iii) has $w_\eps(c_2)>w_\eps(c_1)$ for some $0<c_1<c_2$ and all sufficiently small $\eps>0$.
\end{theorem}
\begin{proof}
If (i) holds, each integrand $H_B(c\omega)$ in \eqref{eq:servicefilter} is nonincreasing in $c$, because $H_B$ is even. Positivity of $\mu$ proves (ii). Finite-state continuous-time Markov chains are mean-square continuous, so (ii) implies (iii).

For the converse, suppose that $H_B(s_2)>H_B(s_1)$ for some $0<s_1<s_2$. Such a pair exists whenever (i) fails: $H_B(0)=b_2/2$ is a global upper bound, so a violating pair cannot require its smaller argument to be zero. For integers $m\to\infty$, use \eqref{eq:generator} with $\eta=1/m$. By Lemma~\ref{lem:cycle}, the corresponding coefficients satisfy
\[
 A_B^{(m)}(c)=2b_1b_2+\frac{b_1}{2}
       \int_0^\infty k_B(u)e^{-d_mcu}\cos(cu)\,du,
 \qquad d_m=\tan(\pi/m)+1/m\longrightarrow0.
\]
Dominated convergence, with dominating function $k_B$, yields
\[
 A_B^{(m)}(s_2)-A_B^{(m)}(s_1)
 \longrightarrow\frac{b_1}{2}\bigl(H_B(s_2)-H_B(s_1)\bigr)>0.
\]
A sufficiently large finite $m$ contradicts (iii). Its intensity is injective, strictly positive and Markov, and all rates are positive. Corollary~\ref{cor:signtest} then gives the final assertion for the same finite $m$ at all sufficiently small traffic levels.
\end{proof}

\begin{corollary}[Mixtures of exponential service distributions]\label{cor:exponential}
Suppose the survival function of $B$ is
\[
 \Pp(B>u)=\int_{(0,\infty)}e^{-qu}\,\nu(dq),
\]
where $\nu$ is a probability measure and $\int q^{-2}\nu(dq)<\infty$. Then
\begin{equation}\label{eq:exponentialfilter}
 H_B(s)=\int_{(0,\infty)}\frac{1}{q^2+s^2}\,\nu(dq),
\end{equation}
and the equivalent conditions of Theorem~\ref{thm:servicecriterion} hold.
\end{corollary}
\begin{proof}
Conditional on rate $q$, the service law is exponential and its multiplier is $(q^2+s^2)^{-1}$. Integrating this nonnegative expression gives \eqref{eq:exponentialfilter}; it is nonincreasing in $s\ge0$. The moment assumption gives $b_2=2\int q^{-2}\nu(dq)<\infty$ and $b_1<\infty$ by Cauchy--Schwarz.
\end{proof}
The representation of completely monotone functions as exponential mixtures is treated in \cite{SSV12}; no such characterization is needed for the direct calculation above. Corollary~\ref{cor:exponential} is only a second-order statement. It does not establish \eqref{eq:ross} at an arbitrary fixed traffic level.

For deterministic $B=b>0$, $H_B(2\pi/b)=0$ and $H_B(3\pi/b)>0$. Thus the criterion fails. The explicit continuous-service theorem shows that this obstruction is not confined to point-mass service distributions.

\section{The variance criterion in light traffic}\label{sec:variance}
For unit service, write $A_F(c)$ for the coefficient $A_1(c)$ of the general environment $F$, and define
\begin{equation}\label{eq:average}
 \overline F_c=\frac1c\int_0^c F(s)\,ds,
 \qquad K(t)=\Cov(F(0),F(t)).
\end{equation}
Boundedness makes these quantities square-integrable. Stationarity implies that $K$ is even, without any reversibility assumption.

\begin{proposition}[Averaged-environment formula]\label{prop:variance}
For every $c>0$,
\begin{equation}\label{eq:variance}
 A_F(c)=\frac12\E\overline F_c^{\,2}
       =\frac{a^2}{2}+\frac12\Var(\overline F_c).
\end{equation}
Consequently
\begin{equation}\label{eq:varianceexpansion}
 w_\eps(c)=\frac{\eps a}{2}
 +\frac{\eps^2}{2}\bigl(a^2+\Var(\overline F_c)\bigr)+R_{\eps,c},
\end{equation}
with the same nonnegative uniform remainder as in \eqref{eq:unitexpansion}. For any fixed $c_1,c_2>0$,
\begin{equation}\label{eq:limitdiff}
 \lim_{\eps\downarrow0}
 \frac{w_\eps(c_2)-w_\eps(c_1)}{\eps^2}
 =\frac12\left(\Var(\overline F_{c_2})-\Var(\overline F_{c_1})\right).
\end{equation}
\end{proposition}
\begin{proof}
By Fubini and stationarity,
\begin{align*}
 \E\overline F_c^{\,2}
 &=\frac1{c^2}\int_0^c\int_0^c\E[F(s)F(t)]\,ds\,dt\\
 &=\frac2{c^2}\int_0^c(c-v)\E[F(0)F(v)]\,dv\\
 &=2\int_0^1(1-u)\E[F(0)F(cu)]\,du.
\end{align*}
This is twice $A_F(c)$. Since $\E\overline F_c=a$, \eqref{eq:variance} follows. The expansion is Theorem~\ref{thm:expansion}, and its uniform $O(\eps^3)$ bound proves \eqref{eq:limitdiff}.
\end{proof}

\begin{corollary}\label{cor:variancecriterion}
If $\Var(\overline F_{c_2})>\Var(\overline F_{c_1})$ for some $0<c_1<c_2$, then $w_\eps(c_2)>w_\eps(c_1)$ for every sufficiently small positive $\eps$. A quantitative sufficient condition is
\begin{equation}\label{eq:varianceeps}
 \frac{\eps M^3}{2(1-\eps M)}
 <\frac12\left(\Var(\overline F_{c_2})-\Var(\overline F_{c_1})\right).
\end{equation}
If \eqref{eq:ross} holds for all sufficiently small $\eps>0$, then $c\mapsto\Var(\overline F_c)$ is nonincreasing.
\end{corollary}
\begin{proof}
Apply Corollary~\ref{cor:signtest} and Proposition~\ref{prop:variance}. For the necessary condition, divide the assumed workload inequality by $\eps^2$ and let $\eps\downarrow0$ for each fixed pair $c_1<c_2$.
\end{proof}
This is a characterization of the second-order ordering, not of the full workload ordering at a fixed traffic level. When two coefficients are equal, the formula does not determine the sign of the difference of the remainders.

\begin{proposition}[A covariance test]\label{prop:derivative}
Assume that $K$ is continuous. Then
\begin{equation}\label{eq:varderivative}
 \frac{d}{dc}\Var(\overline F_c)
 =\frac2{c^3}\int_0^c(2v-c)K(v)\,dv.
\end{equation}
Thus $c\mapsto\Var(\overline F_c)$ is nonincreasing if and only if the integral in \eqref{eq:varderivative} is nonpositive for every $c>0$. In particular, nonincreasing $K$ on $[0,\infty)$ is sufficient.
\end{proposition}
\begin{proof}
From \eqref{eq:variance},
\[
 \Var(\overline F_c)=\frac2{c^2}\int_0^c(c-v)K(v)\,dv.
\]
Differentiate to obtain \eqref{eq:varderivative}. If $K$ is nonincreasing, pair the points $v$ and $c-v$:
\[
 \int_0^c(2v-c)K(v)\,dv
 =\int_0^{c/2}(2v-c)\bigl(K(v)-K(c-v)\bigr)\,dv\le0.
\]
\end{proof}

\begin{corollary}[Reversible environments at second order]\label{cor:reversible}
Let $X$ be a stationary, irreducible, reversible finite-state continuous-time Markov chain, and let $F(t)=f(X_t)$ be bounded and nonnegative. Then $A_F(c)$ is nonincreasing. If $f$ is nonconstant, it is strictly decreasing. In that case, for every fixed $0<c_1<c_2$, there is $\eps_*(c_1,c_2)>0$ such that
\[
 w_\eps(c_1)>w_\eps(c_2),\qquad 0<\eps<\eps_*(c_1,c_2).
\]
\end{corollary}
\begin{proof}
The generator is self-adjoint in the invariant $L^2$ space; see \cite{Kelly79,LevinPeres17} for the reversibility and spectral background. Expand $f-\E f$ in an orthonormal basis of nonconstant eigenfunctions. For positive eigenvalues $\gamma_j$ of the negative generator,
\begin{equation}\label{eq:revcov}
 K(t)=\sum_j\alpha_j^2 e^{-\gamma_j t},\qquad t\ge0.
\end{equation}
It is nonincreasing, so Proposition~\ref{prop:derivative} applies. If $f$ is nonconstant, some $\alpha_j\ne0$, and differentiation of
\[
 A_F(c)=\frac{a^2}{2}+\int_0^1(1-u)K(cu)\,du
\]
shows that $A_F'(c)<0$. For a fixed pair $c_1<c_2$, let $\Delta=A_F(c_1)-A_F(c_2)>0$. Apply the expansion and its nonnegative remainder bound in the direction $w_\eps(c_1)-w_\eps(c_2)$; the same bound as \eqref{eq:epsthreshold}, with $b_1=b_2=1$, supplies $\eps_*$.
\end{proof}
The corollary is a light-traffic statement for each fixed pair of speeds. It neither asserts the all-traffic ordering for arbitrary reversible environments nor gives a necessary condition for that ordering. In the present construction, nonreversibility permits the damped oscillation in \eqref{eq:covariance}, which is excluded by \eqref{eq:revcov}.

\section{Arriving-customer waiting time}\label{sec:palm}
Let $d_\eps(c)$ be the arrival-Palm mean of the prearrival workload, equivalently the mean waiting time before service of a typical customer in the stationary first-come, first-served queue. Standard marked compensation and Palm formulas are given in \cite{Bremaud81,BB03,LP18}. In the present model they give
\begin{equation}\label{eq:palm}
 d_\eps(c)=\frac{\E[\lambda_{\eps,c}(0)V(0)]}{\eps a}.
\end{equation}
Indeed, the expected sum of prearrival workloads over arrivals in $(0,T]$ equals both $T\eps a\,d_\eps(c)$ and $T\E[\lambda_{\eps,c}(0)V(0)]$. Since $\lambda\le\eps M$ and $\E V<\infty$, these quantities are finite.

\begin{proposition}\label{prop:palm}
Under the assumptions of Theorem~\ref{thm:expansion},
\begin{align}\label{eq:palmbalance}
 w_\eps(c)&=\frac{\eps ab_2}{2}+\eps ab_1d_\eps(c),\notag\\
 d_\eps(c_2)-d_\eps(c_1)
 &=\frac{w_\eps(c_2)-w_\eps(c_1)}{\eps ab_1}.
\end{align}
Thus workload, waiting-time and sojourn-time mean comparisons have the same sign when only $c$ changes.
\end{proposition}
\begin{proof}
For bounded service, substitute \eqref{eq:palm} into \eqref{eq:balance}. For a general finite-second-moment service law, truncate the service marks as in the proof of Theorem~\ref{thm:expansion}. Since $V_n\uparrow V$ and $\lambda$ is bounded, monotone convergence gives convergence of $\E[\lambda V_n]$ to $\E[\lambda V]$. Pass to the limit in the bounded-service identity. Subtraction gives the second formula. Mean sojourn time is $d_\eps(c)+b_1$, and $b_1$ does not change with $c$.
\end{proof}
In Theorem~\ref{thm:64}, the arrival mean is $2\eps$ and $b_1=1$, so
\[
 d_\eps(3\pi)-d_\eps(2\pi)>\frac{\eps}{1200}.
\]
For Theorem~\ref{thm:positive}, $a=3$ and $b_1=1$, giving
\[
 d^+_\eps(3\pi)-d^+_\eps(2\pi)>\frac{\eps}{1500}.
\]
The comparisons in Theorem~\ref{thm:many} transfer in the same way. No ordinary-time/Palm identification through PASTA has been used.

\section{Scope and further questions}
The universal stationary workload ordering in Leskel\"a's Problem~1 fails even with a finite irreducible Markov intensity, strictly positive intensities and transition rates, bounded continuous independent service times, and finite exponential workload moments. Positive autocovariance at every lag does not repair the assertion. Conversely, nonincreasing covariance is sufficient for the second-order unit-service comparison, and the service multiplier of Theorem~\ref{thm:servicecriterion} exactly decides a comparison quantified over all bounded stationary environments.

The distinctions between these conclusions are important. The service criterion controls the second-order coefficient; it neither signs a difference of equal-coefficient remainders nor characterizes the full workload ordering at fixed positive traffic. The reversible-chain corollary applies to each fixed pair of speeds at sufficiently small traffic, not simultaneously to every speed pair under one claimed threshold. We do not settle the deterministic-image question in \cite[Problem~2]{Leskela22}, and do not assert a minimum number of states for a counterexample. These remain separate questions from the explicit disproof and the universal second-order characterization proved here.

\appendix
\section{The two-arrival coefficient}\label{app:pairs}
The unit-service kernel also has a direct two-arrival interpretation, as in classical light-traffic calculations \cite{RS89,Blasz95}. Place the arrivals at ages $y$ and $y+v$, $y,v\ge0$. Their joint workload and isolated-workload sum are
\[
 V_2(y,v)=\max\{0,1-y,2-y-v\},\qquad
 S_2(y,v)=(1-y)_++(1-y-v)_+.
\]
For $0\le v\le1$, their interaction is
\[
 V_2(y,v)-S_2(y,v)=
 \begin{cases}
 y,&0\le y\le1-v,\\
 1-v,&1-v\le y\le1,\\
 2-v-y,&1\le y\le2-v,\\
 0,&y\ge2-v.
 \end{cases}
\]
For $v\ge1$ it vanishes. Integrating over $y$ gives
\[
 \frac{(1-v)^2}{2}+v(1-v)+\frac{(1-v)^2}{2}=1-v
 \qquad(0\le v\le1).
\]
This checks the coefficient normalization. The finite-traffic counterexamples require the uniform remainder in Section~\ref{sec:expansion}, not only this two-point calculation.

\section{Exact arithmetic and reproducibility}\label{app:check}
The accompanying Python script uses rational arithmetic for the displayed numerical comparisons and finite insertion regression tests. No empirical data are used. Table~\ref{tab:constants} records the principal certificates. The trigonometric and exponential bounds and the infinite-history queueing arguments are analytic proofs in the paper; a finite regression test is not their formal verification.
\begin{table}[ht]
\caption{Exact lower bounds used in the counterexamples.}\label{tab:constants}
\centering
\begin{tabular}{@{}ll@{}}
\toprule
Quantity & Rational certificate\\
\midrule
64-state damping upper bound & $10192775/205273088<1/20$\\
Endpoint coefficient gap & $>363599/115776720>1/320$\\
64-state normalized workload gap & $>5879/3500000>1/600$\\
128-state normalized workload gap & $>1663/700000>1/500$\\
Repeated-peak gap, multiplied by $L^2$ & $>275899/21000000>1/100$\\
\bottomrule
\end{tabular}
\end{table}
For numerical evaluation, but not for certifying the signs, one may use
\[
 I_d(c)=\Rea\left(\frac{z-1+e^{-z}}{z^2}\right),\qquad z=c(d-\ii).
\]
The source and checker are available in the public repository
\url{https://github.com/yspennstate/Ross-s-second-queueing-conjecture}.

\section*{Acknowledgments}
AI systems developed the mathematical arguments, calculations, verification code and initial manuscript. The author initiated and organized the project. AI research agents developed the cyclic Markov construction and the uniform workload remainder estimate. A separate Claude-based analysis rederived the expansion and its error bound and obtained the averaged-environment variance formulation. ChatGPT developed the bounded-service and positive-transition-rate extensions; a separate ChatGPT analysis derived the 64-state construction and its endpoint estimate. The present synthesis checked that construction, extended the expansion to finite second service moments, and developed the positive-autocovariance example, the repeated-reversal family and the service-law characterization. It also checked the Palm comparison and prepared the manuscript, bibliography and exact arithmetic tests. This work is part of Math Brain, a project aimed at improving AI systems' ability to solve open problems in mathematics.

The author has reviewed the manuscript and its proofs and accepts full responsibility for the mathematical arguments, proofs and conclusions, including any remaining errors.

\section*{Funding}
The author acknowledges support from the European Research Council under the European Union's Horizon Europe programme (grant agreement No.~101041711), the Simons Foundation through the Collaboration on the Mathematical and Scientific Foundations of Deep Learning, Heights Labs, Convex Nexus Capital, and Israel Science Foundation grants 2258/19 and 4101/25.

\section*{Competing interests}
The author holds equity interests in Heights Labs and Convex Nexus Capital.
\section*{Supplementary code}
The files \texttt{verify.py} and \texttt{checks.json}, available in the public repository cited in Appendix~\ref{app:check}, reproduce the rational comparisons and finite insertion tests. The proofs do not depend on external data or a live service.


\begin{thebibliography}{50}
\bibitem{Asmussen03}
S.~Asmussen (2003).
Applied Probability and Queues.
2nd ed. Applications of Mathematics \textbf{51}. Springer, New York.
\href{https://doi.org/10.1007/b97236}{doi:10.1007/b97236}.

\bibitem{BB03}
F.~Baccelli and P.~Br\'emaud (2003).
Elements of Queueing Theory: Palm Martingale Calculus and Stochastic Recurrences.
2nd ed. Applications of Mathematics \textbf{26}. Springer, Berlin.
\href{https://doi.org/10.1007/978-3-662-11657-9}{doi:10.1007/978-3-662-11657-9}.

\bibitem{BR98}
N.~B\"auerle and T.~Rolski (1998).
A monotonicity result for the workload in Markov-modulated queues.
J. Appl. Probab. \textbf{35} 741--747.

\bibitem{Blasz95}
B.~B{\l}aszczyszyn (1995).
Factorial-moment expansion for stochastic systems.
Stochastic Process. Appl. \textbf{56} 321--335.
\href{https://doi.org/10.1016/0304-4149(94)00071-Z}{doi:10.1016/0304-4149(94)00071-Z}.

\bibitem{BFS95}
B.~B{\l}aszczyszyn, A.~Frey and V.~Schmidt (1995).
Light-traffic approximations for Markov-modulated multi-server queues.
Comm. Statist. Stochastic Models \textbf{11} 423--445.
\href{https://doi.org/10.1080/15326349508807354}{doi:10.1080/15326349508807354}.

\bibitem{BY09}
B.~B{\l}aszczyszyn and D.~Yogeshwaran (2009).
Directionally convex ordering of random measures, shot noise fields, and some applications to wireless communications.
Adv. Appl. Probab. \textbf{41} 623--646.

\bibitem{Bremaud81}
P.~Br\'emaud (1981).
Point Processes and Queues: Martingale Dynamics.
Springer Series in Statistics. Springer, New York.

\bibitem{Brumelle71}
S.~L.~Brumelle (1971).
On the relation between customer and time averages in queues.
J. Appl. Probab. \textbf{8} 508--520.
\href{https://doi.org/10.2307/3212174}{doi:10.2307/3212174}.

\bibitem{CCP91}
C.-S.~Chang, X.~Chao and M.~Pinedo (1991).
Monotonicity results for queues with doubly stochastic Poisson arrivals: Ross's conjecture.
Adv. Appl. Probab. \textbf{23} 210--228.

\bibitem{Cox55}
D.~R.~Cox (1955).
Some statistical methods connected with series of events.
J. Roy. Statist. Soc. Ser. B \textbf{17} 129--157.
\href{https://doi.org/10.1111/j.2517-6161.1955.tb00188.x}{doi:10.1111/j.2517-6161.1955.tb00188.x}.

\bibitem{DR84}
D.~J.~Daley and T.~Rolski (1984).
A light traffic approximation for a single-server queue.
Math. Oper. Res. \textbf{9} 624--628.

\bibitem{DR91}
D.~J.~Daley and T.~Rolski (1991).
Light traffic approximations in queues.
Math. Oper. Res. \textbf{16} 57--71.
\href{https://doi.org/10.1287/moor.16.1.57}{doi:10.1287/moor.16.1.57}.

\bibitem{DR94}
D.~J.~Daley and T.~Rolski (1994).
Light traffic approximations in general stationary single-server queues.
Stochastic Process. Appl. \textbf{49} 141--158.

\bibitem{DVJ03}
D.~J.~Daley and D.~Vere-Jones (2003).
An Introduction to the Theory of Point Processes. Vol.~I: Elementary Theory and Methods.
2nd ed. Springer, New York.
\href{https://doi.org/10.1007/b97277}{doi:10.1007/b97277}.

\bibitem{DVJ08}
D.~J.~Daley and D.~Vere-Jones (2008).
An Introduction to the Theory of Point Processes. Vol.~II: General Theory and Structure.
2nd ed. Springer, New York.
\href{https://doi.org/10.1007/978-0-387-49835-5}{doi:10.1007/978-0-387-49835-5}.

\bibitem{GHSY06}
V.~Gupta, M.~Harchol-Balter, A.~Scheller-Wolf and U.~Yechiali (2006).
Fundamental characteristics of queues with fluctuating load.
In \emph{Proceedings of SIGMETRICS/Performance 2006}. ACM, New York.

\bibitem{HL86}
H.~Heffes and D.~M.~Lucantoni (1986).
A Markov modulated characterization of packetized voice and data traffic and related statistical multiplexer performance.
IEEE J. Sel. Areas Commun. \textbf{4} 856--868.
\href{https://doi.org/10.1109/JSAC.1986.1146393}{doi:10.1109/JSAC.1986.1146393}.

\bibitem{Heyman82}
D.~P.~Heyman (1982).
On Ross's conjectures about queues with non-stationary Poisson arrivals.
J. Appl. Probab. \textbf{19} 245--249.
\href{https://doi.org/10.2307/3213936}{doi:10.2307/3213936}.

\bibitem{HS80}
D.~P.~Heyman and S.~Stidham, Jr. (1980).
The relation between customer and time averages in queues.
Oper. Res. \textbf{28} 983--994.
\href{https://doi.org/10.1287/opre.28.4.983}{doi:10.1287/opre.28.4.983}.

\bibitem{HuPan00}
T.~Hu and X.~Pan (2000).
Comparisons of dependence for stationary Markov processes.
Probab. Engrg. Inform. Sci. \textbf{14} 299--315.

\bibitem{Kallenberg17}
O.~Kallenberg (2017).
Random Measures, Theory and Applications.
Probability Theory and Stochastic Modelling \textbf{77}. Springer, Cham.
\href{https://doi.org/10.1007/978-3-319-41598-7}{doi:10.1007/978-3-319-41598-7}.

\bibitem{Kelly79}
F.~P.~Kelly (1979).
Reversibility and Stochastic Networks.
Wiley, Chichester.

\bibitem{Last14}
G.~Last (2014).
Perturbation analysis of Poisson processes.
Bernoulli \textbf{20} 486--513.
\href{https://doi.org/10.3150/12-BEJ494}{doi:10.3150/12-BEJ494}.

\bibitem{LP18}
G.~Last and M.~Penrose (2018).
Lectures on the Poisson Process.
Institute of Mathematical Statistics Textbooks \textbf{7}. Cambridge Univ. Press, Cambridge.
\href{https://doi.org/10.1017/9781316104477}{doi:10.1017/9781316104477}.

\bibitem{Leskela10}
L.~Leskel\"a (2010).
Stochastic relations of random variables and processes.
J. Theoret. Probab. \textbf{23} 523--546.

\bibitem{Leskela22}
L.~Leskel\"a (2022).
Ross's second conjecture and supermodular stochastic ordering.
Queueing Syst. \textbf{100} 213--215.
\href{https://doi.org/10.1007/s11134-022-09824-0}{doi:10.1007/s11134-022-09824-0}.

\bibitem{LV13}
L.~Leskel\"a and M.~Vihola (2013).
Stochastic order characterization of uniform integrability and tightness.
Statist. Probab. Lett. \textbf{83} 382--389.

\bibitem{LevinPeres17}
D.~A.~Levin and Y.~Peres (2017).
Markov Chains and Mixing Times.
2nd ed., with contributions by E.~L.~Wilmer. Amer. Math. Soc., Providence, RI.

\bibitem{Lindley52}
D.~V.~Lindley (1952).
The theory of queues with a single server.
Proc. Cambridge Philos. Soc. \textbf{48} 277--289.

\bibitem{Loynes62}
R.~M.~Loynes (1962).
The stability of a queue with non-independent inter-arrival and service times.
Proc. Cambridge Philos. Soc. \textbf{58} 497--520.
\href{https://doi.org/10.1017/S0305004100036781}{doi:10.1017/S0305004100036781}.

\bibitem{Massey87}
W.~A.~Massey (1987).
Stochastic orderings for Markov processes on partially ordered spaces.
Math. Oper. Res. \textbf{12} 350--367.

\bibitem{MS93}
L.~E.~Meester and J.~G.~Shanthikumar (1993).
Regularity of stochastic processes: A theory based on directional convexity.
Probab. Engrg. Inform. Sci. \textbf{7} 343--360.

\bibitem{Miyazawa94}
M.~Miyazawa (1994).
Rate conservation laws: A survey.
Queueing Syst. \textbf{15} 1--58.
\href{https://doi.org/10.1007/BF01189231}{doi:10.1007/BF01189231}.

\bibitem{MR04}
N.~Miyoshi and T.~Rolski (2004).
Ross-type conjectures on monotonicity of queues.
Aust. N. Z. J. Stat. \textbf{46} 121--131.
\href{https://doi.org/10.1111/j.1467-842X.2004.00318.x}{doi:10.1111/j.1467-842X.2004.00318.x}.

\bibitem{MullerStoyan02}
A.~M\"uller and D.~Stoyan (2002).
Comparison Methods for Stochastic Models and Risks.
Wiley, Chichester.

\bibitem{Neuts79}
M.~F.~Neuts (1979).
A versatile Markovian point process.
J. Appl. Probab. \textbf{16} 764--779.
\href{https://doi.org/10.2307/3213143}{doi:10.2307/3213143}.

\bibitem{Norris97}
J.~R.~Norris (1997).
Markov Chains.
Cambridge Series in Statistical and Probabilistic Mathematics \textbf{2}. Cambridge Univ. Press, Cambridge.

\bibitem{RS89}
M.~I.~Reiman and B.~Simon (1989).
Open queueing systems in light traffic.
Math. Oper. Res. \textbf{14} 26--59.
\href{https://doi.org/10.1287/moor.14.1.26}{doi:10.1287/moor.14.1.26}.

\bibitem{Rolski81}
T.~Rolski (1981).
Queues with nonstationary input stream: Ross's conjecture.
Adv. Appl. Probab. \textbf{13} 603--618.
\href{https://doi.org/10.2307/1426787}{doi:10.2307/1426787}.

\bibitem{Rolski86}
T.~Rolski (1986).
Upper bounds for single server queues with doubly stochastic Poisson arrivals.
Math. Oper. Res. \textbf{11} 442--450.
\href{https://doi.org/10.1287/moor.11.3.442}{doi:10.1287/moor.11.3.442}.

\bibitem{Ross78}
S.~M.~Ross (1978).
Average delay in queues with non-stationary Poisson arrivals.
J. Appl. Probab. \textbf{15} 602--609.
\href{https://doi.org/10.2307/3213122}{doi:10.2307/3213122}.

\bibitem{SSV12}
R.~L.~Schilling, R.~Song and Z.~Vondra\v{c}ek (2012).
Bernstein Functions: Theory and Applications.
2nd ed. De Gruyter Studies in Mathematics \textbf{37}. De Gruyter, Berlin.
\href{https://doi.org/10.1515/9783110269338}{doi:10.1515/9783110269338}.

\bibitem{ShakedShanthikumar07}
M.~Shaked and J.~G.~Shanthikumar (2007).
Stochastic Orders.
Springer Series in Statistics. Springer, New York.

\bibitem{Strassen65}
V.~Strassen (1965).
The existence of probability measures with given marginals.
Ann. Math. Statist. \textbf{36} 423--439.
\href{https://doi.org/10.1214/aoms/1177700153}{doi:10.1214/aoms/1177700153}.

\bibitem{Tchen80}
A.~H.~Tchen (1980).
Inequalities for distributions with given marginals.
Ann. Probab. \textbf{8} 814--827.
\href{https://doi.org/10.1214/aop/1176994668}{doi:10.1214/aop/1176994668}.

\bibitem{Whitt81}
W.~Whitt (1981).
Comparing counting processes and queues.
Adv. Appl. Probab. \textbf{13} 207--220.

\bibitem{Whitt84}
W.~Whitt (1984).
On approximations for queues, III: Mixtures of exponential distributions.
AT\&T Bell Laboratories Technical Journal \textbf{63} 163--175.

\bibitem{Wolff82}
R.~W.~Wolff (1982).
Poisson arrivals see time averages.
Oper. Res. \textbf{30} 223--231.
\href{https://doi.org/10.1287/opre.30.2.223}{doi:10.1287/opre.30.2.223}.

\bibitem{Yaglom87}
A.~M.~Yaglom (1987).
Correlation Theory of Stationary and Related Random Functions. Vols.~I--II.
Springer, New York.

\bibitem{YN71}
U.~Yechiali and P.~Naor (1971).
Queuing problems with heterogeneous arrivals and service.
Oper. Res. \textbf{19} 722--734.
\href{https://doi.org/10.1287/opre.19.3.722}{doi:10.1287/opre.19.3.722}.
\end{thebibliography}
\end{document}